\documentclass[11pt]{article}
\usepackage{amsmath,amsthm,amssymb,amscd}
\usepackage{graphicx}
\usepackage{graphics}
\usepackage{verbatim}

\usepackage{mathrsfs}
\usepackage{color}
\usepackage[top=1in, bottom=1in, left=1.25in, right=1.25in]{geometry}
\usepackage{enumerate}

\newtheorem{theorem}{Theorem}[section]
\newtheorem{proposition}[theorem]{Proposition}

\newtheorem{lemma}[theorem]{Lemma}
\newtheorem{definition}[theorem]{Definition}
\newtheorem{corollary}[theorem]{Corollary}

\newenvironment{claim}[1]{\par\noindent\underline{Claim:}\space#1}{}

\usepackage{thmtools}
\usepackage{thm-restate}

\usepackage[normalem]{ulem}

\begin{document}

\title{Symmetric Set Coloring of Signed Graphs}

\author{Chiara Cappello\thanks{Member of the NRW-Forschungskolleg Gestaltung von flexiblen Arbeitswelten. chiara.cappello@upb.de}, Eckhard Steffen\thanks{
			 Paderborn University, Department of Mathematics, Warburger Str. 100, 33098 Paderborn,
Germany. es@upb.de}}
\date{}

\maketitle

\begin{abstract}
There are many concepts of signed graph coloring which are defined by
assigning colors to the vertices of the graphs.
These concepts usually differ in the number of self-inverse colors used.
We introduce a unifying concept for this kind of coloring 
by assigning elements from symmetric sets to the vertices of the signed graphs.
In the first part of the paper, we study colorings with elements from symmetric sets where the number of self-inverse elements is fixed. 
We prove a Brooks'-type theorem and upper bounds for the corresponding 
chromatic numbers in terms of the chromatic number of the underlying graph. 
These results are used in the second part where we introduce
the symset-chromatic number $\chi_{sym}(G,\sigma)$ of a signed graph
$(G,\sigma)$. We show that the symset-chromatic number gives the 
minimum partition of a signed 
graph into independent sets and non-bipartite antibalanced subgraphs.
In particular, $\chi_{sym}(G,\sigma) \leq \chi(G)$. 

In the final section
	we show that these colorings can also be formalized as $DP$-colorings. 
\end{abstract}

\section{Introduction} \label{Introduction}

A \emph{signed graph} $(G,\sigma)$ is a multigraph $G$ together with a function 
$\sigma : E(G) \rightarrow \{ \pm  \}$, where $\{ \pm \}$ is seen as a multiplicative group.
The function $\sigma$ is called a \emph{signature} of $G$ and $\sigma(e)$ is called the 
\emph{sign} of $e$. An edge $e$ is \emph{negative} if $\sigma(e) = -$ and
it is \emph{positive} otherwise. 
The set of negative edges is denoted by $E_{\sigma}^-$, and 
$E(G) - E_{\sigma}^-$ is the set of positive edges. 
A multigraph $G$ is sometimes called the \emph{underlying graph} of the signed graph $(G,\sigma)$.  

Let $(G',\sigma_{|_{E(G')}})$ be a subgraph of $(G,\sigma)$. 
The sign of $(G',\sigma_{|_{E(G')}})$ is the product of the signs
of its edges. 
A \emph{circuit} is a connected 2-regular graph. It is \emph{positive} if its sign is $+$ and \emph{negative} otherwise. 
A subgraph $(G',\sigma_{|_{E(G')}})$ is \emph{balanced}
 if all circuits in $(G',\sigma_{|_{E(G')}})$ are positive, otherwise it is \emph{unbalanced}. Furthermore, 
negative (positive) circuits are also often 
called unbalanced (balanced) circuits. 
If $\sigma(e)=+$ for all $e \in E(G)$, then $\sigma$ is the \emph{all-positive} signature and 
it is denoted by $\texttt{\bf +}$,
and if $\sigma(e)=-$ for all $e \in E(G)$, then $\sigma$ is the \emph{all-negative} signature and it is denoted by $\texttt{\bf -}$.

A \emph{switching} of a signed graph $(G,\sigma)$ at a set of vertices $X$ defines a signed graph $(G,\sigma')$ which is obtained from $(G,\sigma)$ by reversing the sign of each edge of the edge cut $\partial_G(X)$, 
where $\partial_G(X)$ denotes the set of edges having exactly one end in $X$, 
i.e.~$\sigma'(e) = - \sigma(e)$ if $e \in \partial_G(X)$ and
$\sigma'(e) = \sigma(e)$ otherwise.

 If $X= \{v\}$, then we also say that 
$(G,\sigma')$ is obtained from $(G,\sigma)$ by switching at $v$.
Switching defines an equivalence relation on the set of all signed graphs on $G$. 
We say that $(G,\sigma_1)$ and $(G,\sigma_2)$ are \emph{equivalent} if they can be obtained from each other by a switching at a vertex set $X$.  
We also say that $\sigma_1$ and $\sigma_2$ are equivalent signatures of $G$.
From Harary's \cite{Balance} characterization of balanced signed graphs it
follows that a signed graph $(G, \sigma)$ is balanced if and only if it is
equivalent to $(G,\texttt{\bf +})$. 

A signed graph $(G,\sigma)$ is \emph{antibalanced} if it is equivalent to 
$(G,\texttt{\bf -})$. 
The signed extension $\pm H$ of a graph $H$ 
	is obtained from $H$ by replacing
every edge by two edges, one positive and one negative.

\subsection{Motivation} 

We consider colorings of signed graphs
which are defined by assigning colors to its vertices.

Let $(G,\sigma)$ be a signed graph and $S$ be a set of colors. 
A function $c:V(G) \longrightarrow S$ is a \emph{coloring} of $(G,\sigma)$.
A coloring $c$ is \emph{proper} if 
$c(v) \not= \sigma(e) c(w)$ for each edge $e=vw$. If $(G,\sigma)$
admits a proper coloring with elements from $S$, we say that
$(G,\sigma)$ is $S$-colorable. 

While the coloring-condition for positive edges remains unchanged with respect to the unsigned case, the condition on a negative edge $e=vw$ requires that $c(v) \neq -c(w)$.
It implies that $-s \in S$ for each $s \in S$. 
At this point, the choice of the elements of $S$ has strong consequences on the colorings.
Indeed, two cases have to be distinguished: when $s$ is a non-self-inverse element, that is $s \neq -s$, and when $s$ is a self-inverse element, and so $s = -s$.

The objective of this paper is to define a coloring and a corresponding chromatic number for a signed graph $(G,\sigma)$ which gives a minimum color coded partition of the vertex set
and which does not depend on the number of self-inverse colors which are admitted for
coloring. To achieve this we will first discuss colorings where the number of self-inverse elements is fixed. This somewhat technical part is performed in Sect. \ref{Sec: t-chromatic},
where we determine the chromatic spectrum of signed graphs and prove a Brooks'-type theorem 
for these kinds of coloring. These results are used in Sect. \ref{Sec:symset chromatic},
where we show that the symset chromatic number (which will be defined later in this section)
describes the minimum partition of the signed graph into independent sets and 
antibalanced non-bipartite subgraphs. It follows that this parameter gives
a lower bound on the number of pairwise vertex-disjoint negative circuits of a signed graph.
We further give an 
upper bound for the symset-chromatic number  for a specific class of signed graphs. In the concluding section we show that circular coloring of signed graphs
\cite{Circular, NWZ_2020} is also covered by our approach and that all these colorings can also be formalized as $DP$-coloring. 

\subsection{Colorings and Basic Results}

If a vertex $v$ of $(G,\Sigma)$ is incident to a positive loop, then $(G,\sigma)$ does not have a proper coloring. 
If $v$ is incident to a negative loop, then it has to be colored with a non-self-inverse color,
which somehow counteracts our aforementioned objective. For these reasons, we will consider multigraphs without loops.

Next, we will introduce coloring of signed graphs which covers all approaches of signed graph
coloring which are defined by assigning colors to the vertices of the graph. The sets
of colors will be symmetric sets.

\begin{definition}
	A set $S$ together with a sign ``$-$'' is a \emph{symmetric set} 
	if it satisfies the following conditions:	
	
	1. $s \in S$ if and only if $-s \in S$.
	
	2. If $s=s'$, then $-s=-s'$. 
	
	3. $s = -(-s)$.
\end{definition}

An element $s$ of a symmetric set $S$ is \emph{self-inverse} if $s=-s$. 
A symmetric set with self-inverse elements $0_1, \dots, 0_t$
and non-self-inverse elements $\pm s_1, \dots, \pm s_k$ is
denoted by $S_{2k}^t$. Clearly, $|S_{2k}^t|=t+2k$.

A further natural requirement on signed graph coloring
is that equivalent signed graphs should have the same
coloring properties. 
Let $(G,\sigma')$ be obtained from $(G,\sigma)$ by switching
at a vertex $u$.
If $(G,\sigma)$ admits a proper coloring $c$ with elements from 
$S_{2k}^t$, then $c'$ with $c'(u) = -c(u)$ and $c'(v)=c(v)$ for
$v \not=u$ is a proper coloring of $(G,\sigma')$ with elements 
from $S_{2k}^t$.

\begin{proposition} \label{equivalent}
	Let $(G,\sigma)$ and $(G,\sigma')$ be equivalent signed graphs. Then $(G,\sigma)$ admits a proper 
	$S_{2k}^t$-coloring if and only if $(G,\sigma')$ admits a proper 
	$S_{2k}^t$-coloring.
\end{proposition}

Schweser and Stiebitz \cite{choosable} used the term symmetric
set for subsets $Z \subseteq \mathbb{Z}$ with the property that
$Z = -Z$, where $-Z = \{-z : z \in Z\}$. In case of finite sets
this gives symmetric sets with $t$ self-inverse elements for $t \in \{0,1\}$.  Examples for symmetric 
sets with more than one self-inverse element are subsets $Z'$
of $\mathbb{Z}_{2n}^k$, with $Z'=-Z'$. Here the vectors whose 
entries are either $0$ or $n$ are self-inverse. 

Self-inverse elements someway annul the effect of the sign, so the following proposition naturally holds.

\begin{proposition} \label{proper_coloring_self_inverse}
	Every signed graph $(G,\sigma)$ has a proper $S_0^{\chi(G)}$-coloring, where $\chi(G)$ denotes the chromatic number of the underlying graph $G$. 
\end{proposition}

The main problem in coloring with symmetric sets is that their
cardinality and the number self-inverse elements have the same parity. 
This has some surprising consequences as it can be that the
set of colors has more elements than the vertex set of the graph.
This issue has been addressed in several ways (see e.g.~\cite{SV_21}).

Zaslavsky \cite{Signedgraphcoloring, Colorful} first considered two different sets for coloring signed graphs, $M_{2k}= \{\pm 1, \dots, \pm k\}$ and 
$M_{2k+1}= \{0, \pm 1, \dots, \pm k\}$. 
He worked on the chromatic polynomial by distinguishing the two cases, the 0-free coloring
with elements from $M_{2k}$ and the coloring with elements from $M_{2k+1}$, i.e.~coloring
where $0$ is allowed. 

Based on this coloring, M\'a\v{c}ajov\'a, Raspaud, and \v{S}koviera \cite{ModIntro} introduced the signed chromatic number $\chi_{\pm}(G,\sigma)$
to be the smallest integer $n$ for which $(G,\sigma)$ admits a proper 
$M_{n}$-coloring.

Kang and Steffen \cite{Circular} introduced cyclic coloring of signed graphs and they used cyclic groups $\mathbb{Z}_n$ as the set of colors. The cyclic chromatic number, denoted by $\chi_{mod}(G,\sigma)$, is the smallest integer $n$ such that $(G,\sigma)$ admits a proper coloring with elements of $\mathbb{Z}_n$.

Interestingly, these approaches may not only provide different chromatic numbers for the same signed graphs, but they even have different general bounds.
On one side an antibalanced triangle is 
colorable with $M_2$ by assigning $1$ to all its vertices, but
it is not $\mathbb{Z}_2$-colorable.
On the other side, the signed extension of the complete graph on 4 vertices, $\pm K_4$, has a $\mathbb{Z}_6$-coloring but no $M_6$-coloring. 
Note that here, in both types of coloring
the set of colors contains more elements than the vertex set of the graph.
The reason for this is given by the different number of self-inverse elements allowed: indeed, in the coloring defined by Zaslavsky we can have either 0 or 1 self-inverse element, while cyclic coloring uses 1 or 2 self-inverse elements. 

Let $(G,\sigma)$ be a signed graph and $t \in \{0, \dots \chi(G)\}$ be fixed. 
The \emph{symset $t$-chromatic number} (or $t$-chromatic number for short) 
of $(G,\sigma)$ is the minimum $\lambda_t=t+2k$ for which
$(G,\sigma)$ admits an $S_{2k}^t$-coloring, and it is denoted by $\chi_{sym}^t(G,\sigma)$. 
By Proposition \ref{equivalent}, if $(G,\sigma)$ and 
$(G,\sigma')$ are equivalent, then 
$\chi_{sym}^t(G,\sigma)= \chi_{sym}^t(G,\sigma')$. 
If a graph has symset $t$-chromatic number $\chi_{sym}^t(G,\sigma)=\lambda_t$, we say that $(G, \sigma)$ is \emph{ $\lambda_t$-chromatic}. 

Clearly, Zaslavsky's coloring is equivalent to coloring with elements
from $S^0_{2k}$ or $S^1_{2k}$ and cyclic coloring is equivalent to
coloring with elements from $S^1_{2k}$ and $S^2_{2k}$. 
Consequently, the chromatic
numbers studied in \cite{Circular, ModIntro} can be defined as the
minimum between two specific $t$-chromatic numbers of signed graphs.  

\begin{proposition} \label{first_approaches}
	If $(G,\sigma)$ is a signed graph, then 
	$\chi_{\pm}(G,\sigma) = \min\{\chi_{sym}^0(G,\sigma), \chi_{sym}^1(G,\sigma)\}$ and 
	$\chi_{mod}(G,\sigma) = \min \{\chi_{sym}^1(G,\sigma), \chi_{sym}^2(G,\sigma)\}$. 
\end{proposition}

Thus, for a signed graph $(G, \sigma)$, $\chi_{\pm}(G, \sigma)$ and $\chi_{mod} (G, \sigma)$ depend on the number of self-inverse colors which can be 
used for coloring.
In general, fixing the number of self-inverse elements (instead of choosing from two different cases) causes some issues due to parity. For instance, 
the all-positive complete graph on $2n$ vertices has a proper $S_{2j}^{2(n-j)}$-coloring
for each $j \in \{0,\dots,n\}$, 
and in particular $\chi ^{2j}_{sym}(K_{2n}, \texttt{\bf +}) = 2n$.
However, for each $j \in \{0,\dots,n\}$ it also holds $\chi ^{2j+1}_{sym}(K_{2n}, \texttt{\bf +}) = 2n+1$.

\subsection*{The Symset-Chromatic Number}

Since $S_0^{\chi(G)} \subseteq S_{2k}^t$ for all $t \geq \chi(G)$,
it follows by Proposition \ref{proper_coloring_self_inverse} that every signed graph $(G,\sigma)$ has a proper $S_{2k}^t$-coloring for all $t \geq \chi(G)$. For this reason, we 
assume $t \leq \chi(G)$ in the following.

If an antibalanced subgraph of $(G,\sigma)$
which is induced by a non-self-inverse color is bipartite, then 
the non-self-inverse color can be replaced by two self-inverse colors.
In that case, $(G,\sigma)$ has an $S^t_{2k}$- and an $S^{t+2}_{2(k-1)}$-coloring.

Let $N=\min \{\chi_{sym}^t(G,\sigma) \colon 0 \leq t \leq \chi(G)\}$.
The above examples show that $N$ is not necessarily associated with a
unique symmetric set $S$ for which $(G,\sigma)$ admits a minimum proper $S$-coloring. 
To overcome this problem we define 
$\max_{t} \min \{\chi_{sym}^t(G,\sigma) \colon 0 \leq t \leq \chi(G)\}$
to be the \emph{symset chromatic number} of $(G,\sigma)$, which is denoted by $\chi_{sym}(G,\sigma)$.
Furthermore, we say that an $S_{2k}^t$-coloring is minimum if $\chi_{sym}(G, \sigma)= t+2k$.
By Proposition \ref{equivalent} it follows that equivalent signed graphs have the same symset chromatic number.

\begin{proposition}\label{sym_upper_bound}
	For every signed graph $(G,\sigma)$ it holds $\chi_{sym}(G,\sigma) \leq \chi(G)$. 
	Furthermore, if	$(G,\sigma)$ and $(G,\sigma')$ are equivalent, then 
	$\chi_{sym}(G,\sigma) = \chi_{sym}(G,\sigma')$. In particular, if 
	$(G,\sigma)$ is equivalent to $(G,\texttt{\bf +})$, then 
	$\chi_{sym}(G,\sigma) = \chi(G)$. 
\end{proposition}

In preparation for the study of
the symset-chromatic number in Sect. \ref{Sec:symset chromatic}
we study the symset-$t$-chromatic number in the next section.

\section{The Symset $t$-Chromatic Number} \label{Sec: t-chromatic}

An $S^t_{2k}$-coloring of $(G,\sigma)$ provides some information on the structure
of $(G,\sigma)$.

\begin{proposition} \label{structure}
	If a signed graph $(G,\sigma)$ admits a proper $S_{2k}^t$-coloring $c$, 
	then $c$ induces a partition of $V(G)$
	such that $c^{-1}(0)$ is an independent set in $G$ for every self-inverse
	color $0$, and 
	$(G[c^{-1}(\pm s)], \sigma_{|_{G[c^{-1}(\pm s)]}})$ is an antibalanced subgraph of $(G,\sigma)$ for every non-self-inverse color $\pm s$.  
\end{proposition}

We first prove upper bounds for the symset $t$-chromatic number in terms 
of the chromatic number of the underlying graph. The cases $t \leq 1$ and $t=2$
had been proved in \cite{ModIntro} and \cite{kang2018coloring}, respectively.

\begin{theorem} \label{upperbound_chi(G)}
Let $G$ be a graph with chromatic number $k$. Then for every $t \in \{0, \dots, k\}$: $\chi_{sym}^t(G,\sigma) \leq 2k-t$. Furthermore, $\chi_{sym}^t(\pm G) = 2k-t$ and there are 
simple signed graphs $(H,\sigma_H)$ such that $\chi(H)=k$ and $\chi_{sym}^t(H,\sigma_H)=2k-t$. 
\end{theorem}

\begin{proof}
	Let $c$ be a $k$-coloring of $G$ with colors from 
	$\{0_1,\dots 0_t, s_{t+1}, \dots, s_k\}$.
	This coloring is a coloring of $\pm G$ with colors from  
	$\{0_1,\dots 0_t, \pm s_{t+1}, \dots, \pm s_k\}$. Hence, $\chi_{sym}^t(\pm G) \leq 2k-t$.
	
	If $t=k$, then $\chi_{sym}^t(\pm G) = 2k-t$, since $\chi(G)=k$. 
	
	Let $t < k$ and suppose to the contrary, that $\chi_{sym}^t(\pm G) < 2k-t$. Then, there is a coloring
	with elements $\{0_1,\dots 0_t, \pm s_{t+1}, \dots, \pm s_{l}\}$ and $l < k$. 
	If necessary by switching there is a $2l-t$ coloring of $(G,\sigma)$ which
	only uses colors $\{0_1,\dots 0_t, s_{t+1}, \dots, s_l\}$. This is also an $l$-coloring of $G$, a contradiction. Hence, $\chi_{sym}^t(\pm G) = 2k-t$
	and $\chi_{sym}^t(G,\sigma) \leq 2k-t$, since $(G,\sigma)$ is a subgraph of $\pm G$.

Let $G_k$ be the Turan graph on $k(k-t+1)$ vertices which is the complete
$k$-partite graph with $k$ independent sets of cardinality $k-t+1$.
Thus, $G_k$ contains $k-t+1$ pairwise disjoint copies $H_1, \dots,H_{k-t+1}$
of $K_k$. Let $\sigma$ be a signature on $G_k$ with 
$E_{\sigma}^- = \bigcup_{i=2}^{k-t+1}E(H_i)$. Clearly, $\chi(G_k)=k$ and
therefore, $\chi_{sym}^t(G_k,\sigma) \leq 2k-t$. 

If $t=k$, then $\chi_{sym}^t(G_k,\sigma) = \chi(G_k) = k (=2k-k)$. 
Let $t \in \{0, \dots , k-1\}$ and suppose to the contrary that 
$\chi_{sym}^t(G_k,\sigma) < 2k-t$, say $(G,\sigma)$ is colored with
colors from $\{0_1, \dots, 0_t, \pm s_{t+1} \dots \pm s_{l}\}$ with $l < k$. 

Then, at least $k-t$ vertices of $H_1$ are colored
with pairwise different colors from $\{\pm s_{t+1} \dots \pm s_{l}\}$. Furthermore,
for each $i \in \{2, \dots, k-t\}$ each all-negative copy $H_i$ of $K_k$
contains at least two vertices of the same color of $\{\pm s_{t+1} \dots \pm s_{l}\}$. 
Since for all $2 \leq i < j \leq k-t+1$ each vertex of $H_i$
is connected by a positive edge to every vertex of $H_j$, it follows
that for all $i \not = j$, the multiple used colors in $H_i$ are different from the  multiple used colors in $H_j$. 
Thus, at least $2(k-t)$ pairwise different non-self-inverse colors are needed;
$k-t$ for the 
all-negative copies of $K_k$ and $k-t$ for the coloring of $H_1$. But we have only
$2(l-t)$ non-self-inverse colors and $l<k$, a contradiction. Thus, $\chi_{sym}^t(G_k,\sigma) = 2k-t$.   
\end{proof}

\subsection{Brooks' Type Theorem for the Symset $t$-Chromatic Number}\label{BrooksTheorem}

We are going to prove a Brooks' type theorem for the symset $t$-chromatic number,
which implies Brooks' Theorem for unsigned graphs.  
 
The symset $t$-chromatic number has the same parity as $t$. 
Observe that, if $t = \chi(G) - l$, then
Theorem \ref{upperbound_chi(G)} can be reformulated as
$\chi^t_{sym}(G,\sigma) \leq \chi(G)  +l$. By parity we obtain equality in the following statement.  

\begin{proposition} \label{trivialbound}
	Let $(G,\sigma)$ be a signed graph. If $t=\chi(G)-1$, then $\chi^t _{sym} (G, \sigma)= \chi (G) +1$. 
\end{proposition}

If $G$ is a graph with $\chi(G)=\Delta(G) = t+1$, then 
$\chi^t_{sym}(G,\sigma) = \Delta(G) + 1$ by Proposition \ref{trivialbound}. 
The following Brooks' type statement is the main result of this section.

\begin{theorem}\label{Brooks'adapt}
	Let $G$ be a connected graph and $t \in \{0,\dots,\chi(G)\}$.	\\	
 If $\Delta(G) - t$ is odd, then $\chi^t_{sym}(G,\sigma) \leq \Delta(G) + 1$. \\
 If $\Delta(G) - t$ is even, then  
		$\chi^t_{sym}(G,\sigma) = \Delta(G) + 2$ or $\chi^t_{sym}(G,\sigma) \leq \Delta(G)$. Furthermore, $\chi^t_{sym}(G,\sigma) = \Delta(G) + 2$ if and only if
		\begin{itemize}
			\item $G$ is a complete graph and $t = \chi(G)-1 (= \Delta(G))$ or 
			\item $(G,\sigma)$ is a balanced complete graph or 
			\item $(G,\sigma)$ is a balanced odd circuit or
			\item $(G,\sigma)$ is an  unbalanced even circuit and $t=0$ or
			\item $(G,\sigma)$ is an unbalanced odd circuit and $t=2$. 
		\end{itemize}  	
\end{theorem}

We will prove the statement  
by formulating some propositions, some of which might be of own interest. 

\begin{proposition}\label{CompleteBrooksEven}
Let $K_n$ be the complete graph on $n \geq 3$ vertices and let 
$t \in \{0, \dots ,n\}$. \\
If $\Delta(K_n) -t$ is odd, then $\chi_{sym}^t (K_n,\sigma) \leq \Delta(K_n) +1$.\\
If $\Delta(K_n) -t$ is even, then $\chi^t _{sym} (K_n, \sigma)=\Delta(K_n) +2$ or
 $\chi^t _{sym}(K_n, \sigma) \leq \Delta(K_n)$. Furthermore, 
 $\chi^t _{sym} (K_n, \sigma)=\Delta(K_n) +2$
if and only if $(K_n,\sigma)$ is equivalent to $(K_n,\textbf{+})$ or $t = n-1$.
\end{proposition}

\begin{proof}
If $t=n$ or $t=n-2$, then $\chi_{sym}^t (K_n,\sigma) = \chi(K_n) = n = \Delta(K_n)+1$.
If $t = n-1$, then by Proposition \ref{upperbound_chi(G)}, 
$\chi_{sym}^t (K_n,\sigma) = \chi(K_n)+1 = n+1 = \Delta(K_n)+2$.

Let $t \leq n-3$ and $\chi_{sym}^t (K_n,\sigma) = t + 2k$. Hence, $k \geq 1$. 
First we consider the case when $(K_n,\sigma)$ is not balanced. 
Then it contains an induced antibalanced 
circuit $C_3$ of length 3, which can be colored with one pair of non-self-inverse colors.
Thus, $(K_n-V(C_3),\sigma_{|_{K_n-V(C_3)}})$ can be colored with at most $n-3$ pairwise different colors. Taking the parity into account it follows that 
if $\Delta(K_n) -t$ is even, then 
$\chi^t _{sym} (K_n, \sigma) \leq n-1 = \Delta(K_n)$, and 
if $\Delta(K_n) -t$ is odd, then $\chi^t _{sym} (K_n, \sigma) \leq n = \Delta(K_n) + 1$. 

If $\sigma$ is equivalent to $\textbf{+}$, then any coloring of $(K_n,\textbf{+})$
needs $n$ pairwise different colors. Thus, $\Delta(K_n)-t$ is even if and only if 
$\chi^t_{sym}(K_n, \textbf{+}) = n+1 = \Delta(K_n)+2$. 
\end{proof} 

\begin{proposition}\label{CycleBrooks}
	For each circuit $C_n$ on $n$ vertices: 	
	\begin{itemize}
		\item If $t \in \{1,3\}$, then $\chi^t_{sym} (C_n, \sigma)=3$.
		\item If $t \in \{0,2\}$, then $\chi^t_{sym}(C_n,\sigma) \in \{2,4\}$, and
		$\chi^t_{sym}(C_n,\sigma) = 4$ if and only if 
		\begin{itemize}
			\item $(C_n,\sigma)$ is a balanced odd circuit or
			\item $(C_n,\sigma)$ is an unbalanced even circuit and $t=0$ or
			\item $(C_n,\sigma)$ is an unbalanced odd circuit and  $t=2$. 		
		\end{itemize} 
 	\end{itemize}	
\end{proposition}

\begin{proof} Since we assume that $t \leq \chi(C_n)$, it follows that 
	$t \in \{0,1,2,3\}$, where $t= 3$ only applies if $n$ is odd. In this case, we have $\chi^3_{sym}(C_n,\sigma) = \chi(C_n) = 3$. Furthermore, $\chi^1_{sym}(C_n,\sigma)=3$
	is easy to check. The statements for $t=2$ follow with
	Proposition \ref{trivialbound}.	It is easy to see that 
	$\chi^0_{sym}(C_n,\sigma) \leq 4$ and $\chi^0_{sym}(C_n,\sigma) = 2$
	if and only if $n$ is even and $C_n$ is balanced or $n$ is odd and $C_n$ is unbalanced. 
\end{proof}

The following statement is a standard lemma for coloring.
\begin{lemma} \label{order}
The vertices of a connected graph $G$ can be ordered in a sequence $x_1, x_2,..., x_n$ so that $x_n$ is any preassigned vertex of $G$ and for each $i < n$ the vertex $x_i$ has a neighbor among $x_{i+1},...,x_n$.
\end{lemma}

\begin{lemma}\label{NotRegularBrooks}
Let $(G, \sigma)$ be a simple connected signed graph. If $G$ is not regular, then \\ 
$\chi_{sym}^t (G, \sigma) \leq \begin{cases} \Delta(G) + 1, & \mbox{if }\Delta(G) -t \mbox{ is odd} \\ \Delta(G), & \mbox{if } \Delta(G) -t \mbox{ is even}.
\end{cases}$
\end{lemma}

\begin{proof}
Let $v$ be a vertex having degree $d_G(v) \leq \Delta -1$. 
By Lemma \ref{order}, there exists an ordering of the vertices $x_1,...,x_n$ such that $x_n=v$ and for each $i < n$ the vertex $x_i$ has neighbors among $x_{i+1},..., x_n$. 
We follow this order to color the vertices by using the greedy algorithm. 
We can first use the $t$ self-inverse colors, and then add pairs of non-self-inverse colors when it is necessary.\\
If $\Delta - t = 2n$ is even, then we use exactly $n$ non-self-inverse colors $\pm s$.  Each vertex $x_i$, $i< n$, has at most $\Delta - 1$ neighbors which have been colored previously. Since it also holds $d(x_n) \leq \Delta - 1$, the graph has an $S_{2k}^t$-coloring, with $t+2k= \Delta$.\\
If $\Delta - t$ is odd, then the result follows similarly. 

\end{proof}

For the proof of Theorem \ref{Brooks'adapt} we also use the following lemma.

\begin{lemma} [\cite{LOVASZ}]\label{tool}
Let $G$ be a 2-connected graph with $\Delta(G) \geq 3$ other than a complete graph. Then $G$ contains a pair of vertices $a$ and $b$ at distance 2 such that the graph $G- \{a,b\}$ is connected.
\end{lemma}

\subsubsection{Proof of Theorem \ref{Brooks'adapt} }

\begin{proof}
Propositions \ref{CompleteBrooksEven} and \ref{CycleBrooks} imply that the statement is true for complete graphs and circuits. By Lemma \ref{NotRegularBrooks} it suffices to prove it for non-complete regular graphs with a maximum vertex of degree at least 3. We can also assume that the graph is connected.

Let $(G,\sigma)$ be a signed graph of order $n$ and $0 \leq t \leq \chi(G)$. 
If $\Delta(G) - t$ is odd, then $(G, \sigma)$ can be colored greedily with $\Delta(G) + 1$ colors. Hence, we focus on the case where $\Delta(G)-t$ is even.  
We show that the graph has an $S^t_{2k}$-coloring with $t+2k \leq \Delta(G)$.  

Assume that $(G,\sigma)$ is 2-connected. By Lemma \ref{tool} there are two non-adjacent vertices $a$ and $b$ which have a common neighbor $x$ and $G-\{a,b\}$ is connected.
By possible switching we can assume that $ax$ and $bx$ both are positive. Order the 
vertices of $G$ as in Lemma \ref{order} so that $x_1=a$, $x_2=b$ and $x_n=x$. 
The vertices $x_1$ and $x_2$ can receive the same color since they are not adjacent. The vertices $x_3$,...,$x_{n-1}$ can be colored greedily, indeed each vertex has at most $\Delta(G)-1$ neighbors which are already colored. 
Since two neighbors of $x$ have the same color, there is an element 
of $S^t_{2k}$ which is not used in the neighborhood of $x$. 

Assume now that $(G, \sigma)$ is not 2-connected, that is, there exists a cut vertex $v$. 

Let $H_1$, $H_2$,...,$H_k$ be the components of $G-v$. For each $i\in \{1,...,k\}$,
the subgraph $H_i'=H_i \cup v$ is not regular and $d_{H_i'}(v) < \Delta(H_i')$. Thus, it can be colored by $\Delta(G)$ colors by Lemma \ref{NotRegularBrooks}. By relabeling we can always suppose that $v$ is colored with the same element in each graph, so the entire graph is also $S_{2k}^t$-colorable, with $t+2k=\Delta(G)$. 
\end{proof}

\begin{corollary} [Brooks' Theorem \cite{brooks_1941}] \label{BT_original}
	Let $G$ be a connected graph. If $G$ is neither complete nor an odd circuit, then $\chi(G) \leq \Delta(G)$. 
\end{corollary}

\begin{proof}
	By induction we get $\chi(G) \leq \Delta(G)+1$. Assume that 
	$\chi(G) = \Delta(G)+1$. For $t= \Delta(G)$ it follows 
	by Proposition \ref{trivialbound} that 
	$\chi^t_{sym}(G,\textbf{+}) = \Delta(G) + 2$. Hence, by Theorem \ref{Brooks'adapt},
	$G$ is a complete graph or it is an odd circuit.  
\end{proof}

As a simple consequence of Theorem \ref{upperbound_chi(G)} and Corollary
\ref{BT_original} we obtain the following statement on the signed extension of a graph. 

\begin{corollary} \label{pm_bound}
	Let $G$ be a connected graph. If $G$ is a complete graph or an odd circuit, then $\chi_{sym}^t (\pm G)= \Delta (\pm G)+2-t$. Otherwise $\chi_{sym}^t(\pm G) \leq \Delta(\pm G) - t$.
\end{corollary}

\subsection{Symset $t$-Chromatic Spectrum}

Let $G$ be a graph and $\Sigma(G)$ be the set of its non-equivalent signatures. The \emph{symset $t$-chromatic spectrum} of $G$ is the set $\Sigma_{\chi_{sym}^t}(G):= \{ \chi_{sym}^t(G, \sigma) : \sigma \in \Sigma(G) \}$. We define $m_{\chi_{sym}^t}(G)= min \Sigma_{\chi_{sym}^t}(G)$, and $M_{\chi_{sym}^t}(G)= max  \Sigma_{\chi_{sym}^t}(G)$.

Since $|S_{2k}^t|$ has the same parity as $t$, it follows that
the $t$-chromatic spectrum contains only values of the same parity.  

The question of the 
$t$-chromatic spectrum
of a signed graph was studied for $t \in \{0,1,2\}$ 
in \cite{Spectrum} first. There, it is shown that 
$\Sigma_{\chi_{sym}^0}(G) \cup \Sigma_{\chi_{sym}^1}(G)$
and
$\Sigma_{\chi_{sym}^1}(G) \cup \Sigma_{\chi_{sym}^2}(G)$
are intervals of integers. 

Observe that, if $t=\chi (G)$, then it follows that $\Sigma_{\chi_{sym}^t}(G)= \{t\}$. Hence, we assume $t \leq \chi(G)-1$.

\begin{proposition}\label{minimum}
	Let $G$ be a graph and $t$ a positive integer. Then $m_{\chi_{sym}^t}(G) = t+2$.
\end{proposition}

\begin{proof}	
	Consider the signed graph $(G,\textbf{-})$ and the coloring $c: V(G) \rightarrow S_2^t$ with $c(v)=1$ $\forall v \in V(G)$. This coloring is proper and uses $t+2$ colors. 
	Since $t < \chi (G)$, there exists no signature $\sigma '$ such that $\chi_{sym}^t(G, \sigma')=t$, so $m_{\chi_{sym}^t}(G) = t+2$. 
\end{proof}

A $\lambda_t$-chromatic signed graph $(G,\sigma)$ is $\lambda_t$-\emph{critical} if  $\chi^t_{sym}(G-v,\sigma_{|_{G-v}}) < \lambda_t$ for every $v \in V(G)$. 

\begin{lemma} \label{criticalGraph}
	If $(G, \sigma)$ is a $\lambda_t$-chromatic signed graph with $\lambda_t=t+2k$, then $\chi_{sym}^t(G-v, \sigma_{|_{G-v}}) \in \{t+2k, t+2k-2\}$.
\end{lemma}

\begin{proof}
	Suppose that there exists a vertex $v$ such that $\chi_{sym}^t(G-v, \sigma_{|_{G-v}})\leq t+2k-4$. 
The coloring can be easily extended to $(G, \sigma)$ by adding at most two colors, so $\chi_{sym}^t(G, \sigma) \leq t+2k-2$, which is a contradiction.
\end{proof}

As a consequence, if a signed graph $(G,\sigma)$ is $\lambda_t$-critical, then $\chi_{sym}^t(G-v, \sigma_{|_{G-v}})=\lambda_t -2$ for each $v\in V(G)$. In particular, the following statement holds:

\begin{theorem}\label{critical_subgraphs}
	If $(G, \sigma)$ is a $\lambda_t$-chromatic signed graph with $\lambda_t=t+2k$, then $(G, \sigma)$ has a critical $\lambda_t ^i$-chromatic subgraph for each $\lambda_t ^i = t+2i$, $i \in \{1,...,k\}$.
\end{theorem}

\begin{proof}
	First, we stepwise remove vertices $v$ such that the removal of $v$ does not decrease the symset $t$-chromatic number. The remaining subgraph $(G',\sigma_{|_{G'}})$ is $\lambda_t$-critical.\\
	Secondly, we remove another vertex $w$ from $G'$. Lemma \ref{criticalGraph} implies that this graph has $t$-chromatic number $t+2k-2$. By proceeding as before, we find a critical subgraph with the same $t$-chromatic number. This process can be iterated until we obtain a $\lambda_t ^i$-critical graph, for each $i \in \{1,...,k\}$.
\end{proof}

\begin{theorem}\label{t-ChromaticSpectrum}
	Let $G$ be a graph, then $\Sigma_{\chi_{sym}^t}(G)= \{m_{\chi_{sym}^t(G)}=t+2, t+4,...,t+2k=M_{\chi_{sym}^t}(G)\}$.
\end{theorem}

\begin{proof}
	Let $(G, \sigma)$ be a signature such that $\chi_{sym}^t(G, \sigma)=M_{\chi_{sym}^t}(G)=t+2k$. By Theorem \ref{critical_subgraphs}, we know that for each value of $\lambda_t ^i=t+2i$, where $i \in \{1,...,k\}$, $(G,\sigma)$ has a $\lambda_t^i$-chromatic subgraph $(H, \tau)$. Our aim is to prove that the signature $\tau$ can be extended to a signature $\tau '$ in $G$ such that $\chi_{sym}^t(G, \tau')=t+2i$.\\
	Let $c: V(H) \rightarrow S_{2i}^t$ the $\lambda_t^i$-coloring of $(H, \tau)$. For each edge $uv \in E(G)$
	we define $\tau '$ in the following way:\\
	If $u, v \in V(H)$, $\tau(uv)= \tau '(uv)$ .\\
	If $u, v  \notin V(H)$ or $v\in V(H)$ and $u \notin V(H)$ and $v$ is colored with 1, $\tau '(uv)= -$.\\
	If $v\in V(H)$ and $u \notin V(H)$ and $v$ is not colored with 1, $\tau '(uv)=+$.\\
	By defining $c':V(G) \rightarrow S_{2i}^t$ as $c'(v)=c(v)$ if $v\in V(H)$ and $c'(v)=1$ if $v \notin V(H)$ we obtain a proper $S_{2i}^t$ coloring, so the statement follows.
\end{proof}

\section{The Symset Chromatic Number} \label{Sec:symset chromatic}

Next we skip the constraint on the set of colors given by fixing the value of $t$ and we focus on the symset chromatic number.
If $c$ is an $S^t_{2k}$-coloring, then $c^{-1}(0_j)$ is also called a 
 \emph{self-inverse color class} for each $j \in \{1,\dots,t\}$. Similarly, 
 $c^{-1}(\pm s_i)$ is called a \emph{non-self-inverse color class}
 for $i \in\{1, \dots,k\}$.
Clearly, a self-inverse color class is an independent set of the graph, while a non-self-inverse color class induces an antibalanced subgraph.
If the color classes are induced by a $\lambda_t$-coloring of
$(G,\sigma)$ and $\lambda_t = \chi_{sym}(G,\sigma)$, then any non-self-inverse
color class induces a non-bipartite subgraph of $G$, see Proposition \ref{structure}.

\subsection{The Chromatic Spectrum and Structural Implications} 

The symset chromatic number gives some information on
circuits in the underlying graph $G$ and on the frustration index $l(G, \sigma)$, which is defined as the minimum number of edges which have to be removed from $(G,\sigma)$ to make the graph balanced \cite{Balance}.

\begin{theorem} \label{structure_sign}
	Let $(G,\sigma)$ be a signed graph and  $t, k \geq 0$.
	If $\chi_{sym}(G,\sigma)  = t+ 2k$, then $G$
	has at least $k$ pairwise vertex-disjoint 
	odd circuits, which are unbalanced in $(G,\sigma)$. 
	In particular, $k \leq l(G, \sigma)$.
\end{theorem}

\begin{proof} For $k=0$ there is nothing to prove. So assume $k \geq 1$.
	Let $c$ be an $S_{2k}^t$-coloring of $(G, \sigma)$ and let $S$ be a non-self-inverse color class.
	Since $t$ is maximum, it follows that $\chi(G[S]) > 2$.
	Hence, $G[S]$ is not bipartite.
	Thus, it contains an odd and therefore unbalanced circuit.
	Since this is true for every subgraph induced by a non-self-inverse color class, the statement follows.  
\end{proof}

Furthermore, the bound regarding the frustration index is sharp: the graph in Fig. \ref{fig:FrustrationIndex} has frustration index 2 and it can be easily seen that a minimum coloring requires two non-self-inverse colors.

\begin{figure}[h]
	\centering
	\includegraphics[width=0.35\linewidth]{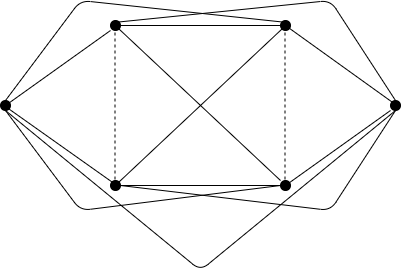}
	\caption[fig:FrustrationIndex]{A graph with frustration index 2 and $\chi _{sym}(G, \sigma)= \chi _{sym}^0 (G, \sigma)=4$.}
	\label{fig:FrustrationIndex}
\end{figure}

Next, we will prove that the symset chromatic spectrum is an interval of integers. 

\begin{theorem}
	The symset chromatic spectrum of a graph $G$ is the interval $\Sigma_{\chi_{sym}} (G) = \{2, \dots,\chi (G)\}$.
\end{theorem}

\begin{proof}
	We proceed by induction on the order of the graph. If $G$ is the $K_1$, then the statement is trivial. Let us remark that it obviously true for bipartite graphs. 
	
	Let $v \in V(G)$ and $G'=G-v$. 
	By induction hypothesis
	$\Sigma_{\chi_{sym}} (G') = \{2, \dots,\chi (G')\}$.
	Let $i \in \{2, \dots,\chi (G')-1\}$ and $\sigma_i'$ be a signature of
	$G'$ such that $\chi_{sym}(G',\sigma_i') = i$.
	Since $i < \chi(G')$
	it follows that $i = t+2k$ and $k \geq 1$. Let $c'$ be an $S^t_{2k}$-coloring
	of $(G',\sigma_i')$. 
	We also assume, by switching, that $c'$ does not use the negative colors. 
	It implies that all the edges connecting vertices in the same non-self-inverse color class are negative.
	
	Extend $\sigma_i'$ to a signature
	$\sigma_i$ of $G$ as follows. Let $vw \in E(G)$. If $c'(w)$ is self-inverse, then 
	let $\sigma_i(vw)=+$ and $\sigma_i(vw)=-$ for otherwise. Since $v$ is connected
	to a non-self-inverse color class by negative edges only, it can be colored 
	with the same color. Thus, $\chi_{sym}(G,\sigma_i) \leq i$. It cannot be smaller,
	since
	for otherwise we would get a symset-coloring of $(G',\sigma_i')$ with less than $i$ colors. 
	Thus,  $\chi_{sym}(G,\sigma_i) = i$ and
	therefore, $\{2, \dots,\chi (G')-1\} \subseteq \Sigma_{\chi_{sym}} (G)$.
	Since $\chi_{sym}(G, \textbf{+})= \chi(G)$, if $\chi(G')=\chi(G)$ the statement follows.\\
	Assume now that $\chi(G')=\chi(G)-1$. 
	We define now a signature $\sigma$ of $G$ such that $\chi_{sym}(G, \sigma)= \chi(G)-1$. \\
	Let $S_1$ and $S_2$ be two of the $\chi(G')$ self-inverse color classes induced by the all-positive signature of $G'$. 
	Define $\sigma'$ in the following way: for each $e=wz$, $\sigma' (e)= -$ if $\{w,z\} \subseteq S_1 \cup S_2$ and $\sigma'(e)=+$ otherwise. $(G', \sigma')$ can be colored with a pair of non-self-inverse colors instead of two self-inverse colors.
	By extending $\sigma'$ to $\sigma$ as before, we obtain a $S_2^{\chi(G)-3}$-coloring, so the statement follows. 
\end{proof}

Let $(G,\sigma)$ be a signed graph with $\chi_{sym}(G,\sigma)= \lambda_t = t+2k$ ($t$ maximum)
and let $c$ be a $\lambda_t$-coloring of $(G,\sigma)$. Let $0_1, \dots,0_t$ be
the self-inverse colors and $\pm s_1, \dots, \pm s_k$ be the non-self-inverse colors. 
Let $I_p = \bigcup_{j=1}^p c^{-1}(0_{i_j})$ be the union of $p$ self-inverse color
classes and $S_q = \bigcup_{j=1}^q c^{-1}(\pm s_{i_j})$ be the union of $q$ non-self-inverse color classes, and $(H_{p,q},\sigma_{p,q}) = (G[I_p \cup S_q], \sigma_{|_{G[I_p \cup S_q]}})$.

\begin{theorem}
	Let $(G,\sigma)$ be a signed graph with $\chi_{sym}(G,\sigma)= \lambda_t = t+2k$ and let $c$ be a $\lambda_t$-coloring of $(G,\sigma)$.
	Then $\chi_{sym}(H_{p,q}, \sigma_{p,q})) = \chi_{sym}^p(H_{p,q}, \sigma_{p,q})) = p + 2q$,
	for each $p \in \{0, \dots, t\}$ and $q \in \{0, \dots, k\}$.
\end{theorem}
\begin{proof}
	By the coloring $c$ of $(G,\sigma)$ we have that 
	$\chi_{sym}(H_{p,q}, \sigma_{p,q})) \leq p + 2q$. However, if there would be 
	a better coloring with less colors or one with the same number of colors but more
	self-inverse colors, then there would be a better coloring for $(G,\sigma)$,
	a contradiction. 
\end{proof}

For $(p,q)=(t,0)$ and $(p,q)=(0,k)$ we obtain the following corollary. 

\begin{corollary}
	Let $(G,\sigma)$ be a signed graph with $\chi_{sym}(G,\sigma)= \lambda_t = t+2k$ and let $c$ be a $\lambda_t$-coloring of $(G,\sigma)$. Then $(G,\sigma)$ can be partitioned into two induced subgraph $(H_1,\sigma_1)$ and 
	$(H_2,\sigma_2)$, such that $\chi_{sym}(H_1,\sigma_1) = t = \chi(H_1)$ and
	$\chi_{sym}(H_2,\sigma_2) = 2k = \chi_{sym}^0(H_2,\sigma_2)$.
\end{corollary}

We conclude with the following statement.

\begin{theorem}
	Let $(G,\sigma)$ be a signed graph. 
	Then $\chi_{sym}(G,\sigma)= \lambda_t = t+2k$ if and only if 
	$(G,\sigma)$ can be partitioned into $t$ independent sets and $k$ non-bipartite antibalanced subgraphs such that $t+2k$ is minimum and
	$t$ maximum for such a partition.  
\end{theorem}

\begin{proof}
	Clearly, each self-inverse color class induces an independent set in $(G,\sigma)$ and
	each non-self-inverse color class an antibalanced subgraph $(H,\gamma)$.
	Since $t$ is maximum it follows that $(H,\gamma)$ is not bipartite.
	
	On the other side, assume that $(G,\sigma)$ has a partition into $t'$ independent sets
	and $k'$ antibalanced subgraphs with $t'+2k'$ minimum.
	Then there is a partition $P$  with the maximum number of independent sets.
	Let $t$ be this number. Then $P$ has $k = \frac{1}{2}(t'+2k'-t)$
	non-bipartite antibalanced subgraphs.  Thus, $(G,\sigma)$ has an $S^t_{2k}$-coloring and therefore, 
	$\chi_{sym}(G,\sigma)= t+2k$.
\end{proof}

\subsubsection*{Upper Bounds for the Symset Chromatic Number}

By definition, $\chi_{sym}(G,\sigma) \leq \chi(G)$ and, therefore, 
Brooks' Theorem can easily be extended to the symset chromatic number.
Indeed, if $\chi_{sym}(G,\sigma) \not = \chi(G)$, then 
$\chi_{sym}(G,\sigma) \leq \Delta(G)-1$ unless $G$ is complete or an odd circuit. 
This statement can be further improved if there is a small non-self-inverse color class.

\begin{theorem} \label{Thm:small_Delta_bound}
Let $(G, \sigma)$ be a signed graph and $\chi _{sym}(G, \sigma)= \lambda_t = t + 2k < \chi(G)$. If there exists a $\lambda_t$-coloring with a non-self-inverse color class of cardinality 3, then $\chi_{sym}(G, \sigma) \leq \Delta (G) - k +1$.
\end{theorem}

\begin{proof} Among all $\lambda_t$-colorings of $(G,\sigma)$ which have 
	a non-self-inverse color class with precisely three vertices choose 
	coloring $c$ with a maximum number of vertices in the
	union of the self-inverse color classes and then choose the non-self-inverse 
	$S_1, \dots, S_k$, such that $|S_1|$ is maximum, according to the choice of
	$S_1, \dots, S_i$ choose $S_{i+1}$ such that 
	$|S_{i+1}|$ is maximum.
	Since every non-self-inverse color class has at least three vertices we can assume
	that $S_k = T$. 

	Clearly, $G[T]$ is a triangle. 
	Note that by the choice of $c$ every vertex of $T$ is connected to 
	each self-inverse color class by an edge and to each non-self-inverse color class by a positive and a negative edge. 
	It implies that each vertex $v \in T$ has $d_G(v) \geq t+2k$.
	
	We will show that there is
	a vertex $v \in T$ with $d_G(v) \geq t+3k -1$.

	Let $S=c^{-1}(\pm s)$ be a non-self-inverse color class. Then 
	$c$ induces an $S^0_4$-coloring of 
	$(G[S \cup T], \sigma_{|_{G[S \cup T]}})$, and we
	can assume that all edges of $G[S]$ and $G[T]$ are negative. 
	Furthermore, each vertex
	of $T$ has degree at least 4 in $G[T \cup S]$ and $d_{G[S \cup T]}(T) \geq 6$.
	We will show that there are at least nine edges between $T$ and $S$.
	Let $V(T)=\{v_1,v_2,v_3\}$ and we assume that $d_{G[T\cup S]}(v_1) \leq  d_{G[T\cup S]}(v_2) \leq d_{G[T\cup S]}(v_3)$.

\begin{claim}
$d_{G[S \cup T]}(T) \geq 9$.
\end{claim}

Proof of the claim. 
Suppose to the contrary that the claim is not true. Then $d_{G[T\cup S]}(v_1) = 4$
and $ 8 \leq d_{G[T\cup S]}(v_2) + d_{G[T\cup S]}(v_3) \leq 10$. Thus,
$d_{G[T\cup S]}(v_2) \leq 5$. Let $\{w_1,w_2\}$ be the neighbors of $v_1$ in $S$. We assume that
$v_1w_1$ is negative and $v_1w_2$ is positive. 

Suppose that $w_2$ is not a neighbor of $v_i$, $i \in \{2,3\}$. Then $w_2$ and $v_i$ can be colored with one self-inverse color, $v_j$ ($j \not= 1,i$) can be colored with 
another self-inverse color and $v_1$ with color $s$, since it is connected by a negative edge to its second neighbor in $S$. 
Thus, $w_2$ is also neighbor of $v_2$ and $v_3$. 
By switching at $T$ we deduce that $w_1,w_2$ are both neighbors of $v_2$ and of $v_3$.

Let $d_{G[T\cup S]}(v_2) = 5$. Hence, $d_{G[T\cup S]}(v_3) = 5$. Let $w_3$ be the third
neighbor of
$v_2$ in $S$. By possible switching at $T$ we can assume that $v_2w_3$ is negative.
If $G[\{v_3,w_1,w_2\}]$ is bipartite, then color it with two (new) self-inverse colors 
and $v_1, v_2$ with color $s$ to obtain an $S^2_2$-coloring of $G[T\cup S]$, a contradiction.

Thus, $w_1w_2 \in E(G)$ (indeed in $E_{\sigma}^-$) and $G[\{v_3,w_1,w_2\}]$ is a triangle. 
Furthermore, $G[\{v_1,w_1,w_2\}]$ is a balanced triangle. Suppose that  
$G[\{v_2,w_1,w_2\}]$ is anti-balanced, then $v_2$ can be colored with $\pm s$ and
$v_1, v_3$ with two self-inverse colors to obtain an $S^2_2$-coloring of $G[T\cup S]$, a contradiction. By possible switching we analogously argue for $G[\{v_3,w_1,w_2\}]$
and hence, $G[\{v_i,w_1,w_2\}]$ is a balanced triangle for each $i \in \{1,2,3\}$.
Since $w_1w_2$ is negative, precisely one of the remaining two edges is positive. 
If one of $w_1,w_2$, say $w_1$ is incident to three positive edges $v_1w_1, v_2w_1, v_3w_1$,
then color $w_1$ and $w_4$ -the third neighbor of $v_3$ in $S$- with two self-inverse colors
and the remaining vertices with color $s$ to obtain an $S^2_2$-coloring of $G[T\cup S]$, a contradiction. 

Hence, $G[T \cup \{w_1,w_2\}]$ is a complete signed subgraph $(H_5,\sigma_5)$ of $(G,\sigma)$.
Clearly, all edges within two vertices
of $T$ are negative and all edges between two vertices of $S$ are negative.
We will discuss the following distribution of positive and negative edges:
$v_1w_2, v_2w_2, v_3w_1$  are positive and all other edges in $(H_5,\sigma_5)$ are negative. Furthermore,
we can assume that $v_2w_3$ is negative (see Figure \ref{fig:H_5}). The argumentation for other distributions is similar.  

If $w_3 = w_4$ and $v_3w_3$ is negative or $w_3 \not = w_4$,
then color $v_3$ and $w_2$ with two self-inverse colors and the remaining vertices with color $s$ to obtain the desired contradiction. 

(*) If $w_3 = w_4$ and $v_3w_3$ is positive, then color $v_3, w_2$ with two self-inverse colors 
and the remaining vertices with color $s$ to obtain an $S^2_2$-coloring of $G[T\cup S]$, 
which is the desired contradiction and finishes the proof of this case.  

It remains to consider the case when $d_{G[T \cup S]}(v_2) = 4$.
We analogously deduce that $(G,\sigma)$ contains $(H_5,\sigma_5)$. 
If $d_{G[T \cup S]}(v_3) = 4$, then $w_1,w_2$ is a bipartite cut in $G[T \cup S]$
and we easily get an $S^2_2$-coloring of $G[T \cup S]$. If 
$d_{G[T \cup S]}(v_3) = 5$, we similarly argue as above by discussing the edge $v_3w_3$
instead of $v_2w_3$. If $d_{G[T \cup S]}(v_3) = 6$, then we may assume that
$v_3w_3$ is negative. However, the coloring given in (*) works here as well and the proof of the claim is finished.

\begin{figure}[h]
	\centering
	\includegraphics[width=0.4\linewidth]{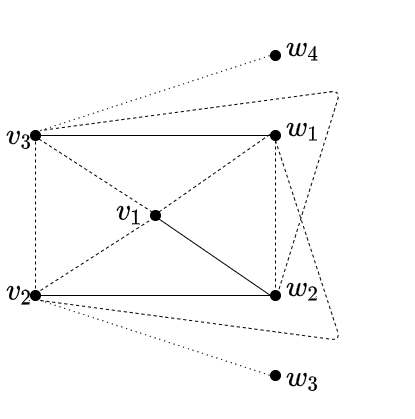}
	\caption[fig:H_5]{The graph $(G[T \cup S], \sigma|_{G[T \cup S]})$, with dotted edges negative and small dotted edges undefined. }
	\label{fig:H_5}
	\end{figure}

Since $(G[T], \sigma_{|_{G[T]}})$ is connected to $t$ self-inverse color classes and $k-1$ non-self-inverse color classes, it holds that $d_G(T) \geq 3t + 9(k-1)$. 
Seeing that $T$ only contains three vertices, each of degree 2 in $G[T]$, it follows that there exists $v \in T$ such that $d_G(v) \geq t + 3(k-1) +2 = t+3k -1$.

\end{proof}

\section{Concluding Remarks on Variants of Coloring Parameters of Signed Graphs} 

\subsection{Circular Coloring} \label{Sec: Variants}

Circular coloring is a well studied refinement of ordinary coloring of graphs.
Here the set of colors is provided with a (circular) metric.
Kang and Steffen \cite{Circular} used elements of cyclic groups as colors for their definition 
of $(k, d)$-coloring of a signed graph $(G,\sigma)$. 
For positive integers $k, d$ with $k \geq 2d$, a $(k,d)${\it -coloring}
of a signed graph $(G, \sigma)$ is a map $c: V(G) \rightarrow \mathbb{Z}_k$ such that for each edge $e=vw$, $|c(v)-\sigma (e) c(w)| \geq d \mod k$.
Hence, this coloring is a specific $S^1_{2k'}$-coloring if $k=2k'+1$, and a specific $S^2_{2(k'-1)}$-coloring if $k=2k'$. 

Naserasr, Wang and Zhu \cite{NWZ_2020} generalized circular coloring 
of graphs to signed graphs as follows. For 
$i,j \in \{0,1, \dots,p-1\}$, the modulo-$p$ distance between $i$ and $j$
is $d_{(\bmod p)}(i,j) = \min\{|i-j|, p-|i-j|\}$. For an even integer $p$,
the antipodal color of $x \in \{0,1,\dots,p-1\}$ is 
$\overline{x} = x + \frac{p}{2} \mod p$.

Let $p$ be an even integer and $q \leq \frac{p}{2}$ be a positive integer. 
A $(p,q)$-coloring of a signed graph $(G,\sigma)$   is a 
mapping $f \colon V(G) \rightarrow \{0,1, \dots,p-1\}$ such that for each positive edge $xy$,
$d_{(\bmod p)}(f(x),f(y)) \geq q$, and for each negative edge $xy$,
$d_{(\bmod p)}(f(x),\overline{f(y)}) \geq q$. 
Now it is easy to see that this defines a specific $S^0_p$-coloring of $(G,\sigma)$.

\subsection{DP-Coloring} \label{Sec: Variants}

In this subsection, we show that coloring of signed graphs with
elements from a symmetric set can be described as special 
$DP$-coloring. The $DP$-coloring was introduced
for graphs by Dvo\v{r}\'{a}k and Postle \cite{DP_Coloring_Def} under the name correspondence coloring. We follow 
Bernshte\v{\i}n, Kostochka, and Pron \cite{BKP_2017} 
and consider multigraphs. 

Let $G$ be a multigraph. A \emph{cover} of $G$ is a pair $(L,H)$, where $L$ is an assignment 
of pairwise disjoint sets to the vertices of $G$ and $H$ is the graph with vertex set
$\bigcup_{v \in V(G)} L(v)$ satisfying the following conditions:

\begin{enumerate}
	\item $H[L(v)]$ is an independent set for each $v \in V(G)$.
	\item For any two distinct vertices $v, w$ of $G$ the set of edges between $L(v)$ and
	$L(w)$ is the union of $\mu_G(v,w)$ (possible empty) matchings, where $\mu_G(v,w)$
	denotes the number of edges between $v$ and $w$ in $G$. 
\end{enumerate}

An \emph{ $(L,H)$-coloring} of $G$ is an 
independent transversal $T$ of cardinality $|V(G)|$ in $H$, i.e.~for each vertex $v \in V(G)$ exactly one vertex of $L(v)$ belongs to $T$ and $H[T]$ is edgeless. We also say that
$G$ is $(L,H)$\emph{-colorable}. 

Let $t,k$ be positive integers and $L^t_{2k} = \{s_1,\dots,s_t, r_0, \dots, r_{2k-1}\}$.
An \emph{ $(L,H^t_{2k})$-cover} of a signed multigraph is a cover of $(G,\sigma)$ with 
$L(v) = L^t_{2k}$ for each vertex $v \in V(G)$ and $H^t_{2k}$ satisfies the following conditions: 

\begin{enumerate}
	\item $H^t_{2k}[L(v)]$ is an independent set for each $v \in V(G)$.
	\item If there is no edge between $u$ and $w$, then $E_{H^t_{2k}}(L(u),L(w)) = \emptyset$. 
	\item For each edge $e$ between $u$ and $w$ we associate a perfect matching $M_e$ of $E_{H^t_{2k}}(L(u),L(w))$ with the property that, if $e$ is a positive edge, then 
	$M_e = \{(q,u)(q,w) \colon q \in L^t_{2k}\}$ and if 
	$e$ is a negative edge, then $M_e$ is a perfect matching 
	of $E_{H^t_{2k}}(L(u),L(w))$ which consists of the edges 
	$(s_i,u)(s_i,w)$ for each $i \in \{1,\dots,t\}$ and
	$(r_j,u)(r_{j+k},w)$ for each $j \in \{0,\dots,2k-1\}$, where the indices are added mod $2k$.  
\end{enumerate}

It is easy to see that a signed graph is
$S^t_{2k}$-colorable if and only if it 
$(L,H^t_{2k})$-colorable.
The associated chromatic numbers are to be defined accordingly. 

If we consider coloring of signed graphs we can restrict to 
multigraphs with edge multiplicity at most 2, 
since more than one 
positive and one negative edge between two vertices do 
not have any effect on the coloring properties of the multigraph. 
That is, if we consider $(L,H^t_{2k})$-cover of a
signed extension $\pm G$ of a graph $G$, then $H^t_{2k}[E_H(L(u),L(w)]$ 
is a 2-regular 
multigraph whose components are digons and circuits of lengths 4,
for any two adjacent vertices $u, v$ of $G$.

However, the DP-coloring approach allows further flexibility and
generalizations. For instance, 
DP-coloring is considered in the more general context of gain graphs in a short note of Slilaty \cite{Slilaty_21},
where the corresponding chromatic polynomials are defined. 

\subsubsection*{Statements}
On behalf of all authors, the corresponding author states that there is no conflict of interest. 
The manuscript has no associated data.

\bibliography{Survey_References}{}

\begin{thebibliography}{10}

\bibitem{BKP_2017}
A.~Y. Bernshte\u{\i}n, A.~V. Kostochka, and S.~P. Pron.
\newblock On {DP}-coloring of graphs and multigraphs.
\newblock {\em Sibirsk. Mat. Zh.}, 58(1):36--47, 2017.

\bibitem{brooks_1941}
R.~L. Brooks.
\newblock On colouring the nodes of a network.
\newblock {\em Mathematical Proceedings of the Cambridge Philosophical
  Society}, 37(2):194–197, 1941.

\bibitem{DP_Coloring_Def}
Z.~Dvo\v{r}\'{a}k and L.~Postle.
\newblock Correspondence coloring and its application to list-coloring planar
  graphs without cycles of lengths 4 to 8.
\newblock {\em J. Combin. Theory Ser. B}, 129:38--54, 2018.

\bibitem{Balance}
F.~Harary.
\newblock On the notion of balance of a signed graph.
\newblock {\em Michigan Math. J.}, 2:143--146, 1953.

\bibitem{kang2018coloring}
Y.~Kang.
\newblock {\em Coloring of signed graphs}.
\newblock PhD thesis, Universit{\"a}t Paderborn, 2018.

\bibitem{Spectrum}
Y.~Kang and E.~Steffen.
\newblock The chromatic spectrum of signed graphs.
\newblock {\em Discrete Math.}, 339(11):2660--2663, 2016.

\bibitem{Circular}
Y.~Kang and E.~Steffen.
\newblock Circular coloring of signed graphs.
\newblock {\em J. Graph Theory}, 87(2):135--148, 2018.

\bibitem{LOVASZ}
L.~Lovász.
\newblock Three short proofs in graph theory.
\newblock {\em J. Combin. Theory Ser B}, 19(3):269--271, 1975.

\bibitem{ModIntro}
E.~M\'{a}\v{c}ajov\'{a}, A.~Raspaud, and M.~\v{S}koviera.
\newblock The chromatic number of a signed graph.
\newblock {\em Electron. J. Combin.}, 23(1):Paper 1.14, 10, 2016.

\bibitem{NWZ_2020}
R.~Naserasr, Z.~Wang, and X.~Zhu.
\newblock Circular chromatic number of signed graphs.
\newblock {\em Electron. J. Combin.}, 28, 06 2021.

\bibitem{choosable}
T.~Schweser and M.~Stiebitz.
\newblock Degree choosable signed graphs.
\newblock {\em Discrete Math.}, 340(5):882--891, 2017.

\bibitem{Slilaty_21}
D.~Slilaty.
\newblock Coloring permutation-gain graphs.
\newblock {\em Contrib. Disc. Math.}, 16:47--52, 2021.

\bibitem{SV_21}
E.~Steffen and A.~Vogel.
\newblock Concepts of signed graph coloring.
\newblock {\em European J. Combin.}, 91:103226, 19, 2021.

\bibitem{Signedgraphcoloring}
T.~Zaslavsky.
\newblock Signed graph coloring.
\newblock {\em Discrete Math.}, 39(2):215--228, 1982.

\bibitem{Colorful}
T.~Zaslavsky.
\newblock How colorful the signed graph?
\newblock {\em Discrete Math.}, 52(2-3):279--284, 1984.

\end{thebibliography}
\addcontentsline{toc}{section}{References}
\bibliographystyle{abbrv}

\end{document}